\documentclass[reqno]{amsart}

\usepackage{graphicx,subfigure,color}
\usepackage{mathtools}

\numberwithin{equation}{section}

\usepackage[latin1]{inputenc}
\usepackage[english]{babel}

\usepackage{amsmath,amsthm,amsfonts,latexsym,amssymb}
\usepackage[colorlinks]{hyperref}
\hypersetup{linkcolor=blue,citecolor=blue,filecolor=black,urlcolor=blue}
\usepackage{comment}

\usepackage{color}
\usepackage[dvipsnames]{xcolor}
\definecolor{darkgreen}{rgb}{0,0.7,0.1}

{ \theoremstyle{plain}
\newtheorem{theorem}{Theorem}[section]

\newtheorem{lemma}[theorem]{Lemma}
\newtheorem{corollary}[theorem]{Corollary}
  \theoremstyle{remark}
\newtheorem{remark}[theorem]{Remark}
  \theoremstyle{definition}

}

\def\R{\mathbb{R}}

\begin{document}
\subjclass[2010]{35J62, 35B05, 35A24, 35B09, 34B18.}

\keywords{Lorentz-Minkowski mean curvature operator, Shooting method, 
Existence and multiplicity, Oscillating solutions, Neumann boundary conditions.}

\title[Positive solutions for the Minkowski-curvature equation]{Positive radial solutions for the Minkowski-curvature equation  with Neumann boundary conditions}

\author[A. Boscaggin]{Alberto Boscaggin}
\address{Alberto Boscaggin\newline\indent 
Dipartimento di Matematica ``Giuseppe Peano''
\newline\indent
Universit\`a di Torino
\newline\indent
via Carlo Alberto 10, 10123 Torino, Italia}
\email{alberto.boscaggin@unito.it}

\author[F. Colasuonno]{Francesca Colasuonno}
\address{Francesca Colasuonno\newline\indent
Dipartimento di Matematica ``Giuseppe Peano''
\newline\indent
Universit\`a di Torino
\newline\indent
via Carlo Alberto 10, 10123 Torino, Italia}
\email{francesca.colasuonno@unito.it}

\author[B. Noris]{Benedetta Noris}
\address{Benedetta Noris
\newline \indent Laboratoire Ami\'enois de Math\'ematique Fondamentale et Appliqu\'ee\newline\indent
Universit\'e de Picardie Jules Verne\newline\indent
33 rue Saint- Leu, 80039 AMIENS, France}
\email{benedetta.noris@u-picardie.fr}

\date{\today}

\begin{abstract}
We analyze existence, multiplicity and oscillatory behavior of positive radial solutions to a class of quasilinear equations governed by the Lorentz-Minkowski mean curvature operator. The equation is set in a ball or an annulus of $\mathbb R^N$, is subject to homogeneous Neumann boundary conditions, and involves a nonlinear term on which we do not impose any growth condition at infinity. The main tool that we use is the shooting method for ODEs. 
\end{abstract}

\maketitle
\begin{quotation}
{\sl Dedicated to Professor Patrizia Pucci, on the occasion of her 65th birthday, with great esteem.} 
\end{quotation}

\section{Introduction}
In this paper we deal with the existence and multiplicity of solutions to the nonlinear boundary value problem
\begin{equation}\label{eq:main}
\left\{
\begin{array}{ll}
\vspace{0.1cm}
\displaystyle{-\mathop{\rm div}\left(\frac{\nabla u}{\sqrt{1-|\nabla u|^2}}\right) = f(u)} & \mbox{ in } \Omega \\
\vspace{0.1cm}
u > 0 & \mbox{ in } \Omega \\
\partial_\nu u = 0 & \mbox{ on } \partial\Omega, \\
\end{array}
\right.
\end{equation}
where $\nu$ is the outer unit normal of $\partial \Omega$ and $\Omega\subset\mathbb R^N$ ($N\geq1$) is a radial domain which can be either an annulus 
$$
\Omega=\mathcal A(R_1,R_2):=\{ x \in \mathbb{R}^N \, : \, R_1 < \vert x \vert < R_2 \},\quad 0<R_1<R_2<+\infty,
$$
or a ball
$$
\Omega=\mathcal B(R_2):=\{ x \in \mathbb{R}^N \, : \vert x \vert < R_2 \},\quad 0<R_2<+\infty.
$$ 
Throughout the paper, in order to treat simultaneously the cases of the annulus and of the ball, we use the convention $R_1=0$ when $\Omega=\mathcal B(R_2)$.
We are interested in radial, $C^2(\overline{\Omega})$ solutions of \eqref{eq:main}, thus writing, with the usual abuse of notation, $u(x) = u(r)$ for $r = \vert x \vert$.

The nonlinear differential operator appearing in \eqref{eq:main} is usually meant as a mean curvature operator in the Lorentz-Minkowski space
and it is of interest in Differential Geometry and General Relativity \cite{BaSi-8283,Ecker,Ge-83}; it also appears in the nonlinear theory of Electromagnetism, where it is usually referred to as Born-Infeld operator \cite{BoCoFo-pp,BdAP-16,BoIa-pp}. In recent years, it has become very popular among specialists in Nonlinear Analysis, and various existence/multiplicity results for the associated boundary value problems are available, both in the ODE and in the PDE case, possibly in a non-radial setting (see, among others, \cite{Az-14,Az-16,BeJeMa-09,BeJeTo-13,BeJeTo-13a,BeMa-07,BoFe-pp,BoGa-ccm,CoCoObOm-12,CoelhoCorsatoRivetti,CoObOmRi-13,DW,Ma-13} and the references therein).

In this paper, we focus on the Neumann boundary value problem \eqref{eq:main}. Notice that, differently from the Dirichlet problem, solutions cannot exist if $f$ has constant sign; therefore, we are led to assume that $f$ has a zero $s_0 > 0$ (and, thus, $u \equiv s_0$ is a constant solution to \eqref{eq:main}). 
As a prototype nonlinearity, we can think at the following difference of pure powers
\begin{equation}\label{eq:prototype}
f(s)=s^{q-1}-s^{r-1} \quad \mbox{with } 2\leq r<q<+\infty,
\end{equation}
which indeed satisfies $f(s_0) = 0$ for $s_0 = 1$.

In such a situation, it has been shown in some recent papers \cite{BonheureGrossiNorisTerracini2015, bonheure2016multiple,BNW,ABF2,
ABF,ma2016bonheure} that non-constant positive Neumann solutions to the semilinear equation 
$$
-\Delta u = f(u) 
$$
can be provided and characterized in terms of the intersection with the constant solution $u \equiv s_0$. More precisely, in \cite{ABF2,ABF}
it is proved that, if $f'(s_0)$ is greater than the $k$-th eigenvalue of the radial Neumann problem for $-\Delta u = \lambda u$, then
a solution with $u(R_1) < s_0$, having exactly $k$ intersections with $s_0$ always exists (incidentally, let us recall that this provides a positive answer to the conjecture raised in \cite{BNW}). Moreover, a solution with $u(R_1) > s_0$ and having exactly $k$ intersections with $s_0$ also exists, if we further assume that $f$ satisfies suitable sub-criticality growth conditions at infinity (that is, in the case of the ball, $q < 2^*$ for the prototype nonlinearity \eqref{eq:prototype}). 
We note in passing that in \cite{ABF2,ABF,ColasuonnoNoris,CNproc} a similar analysis is performed for quasilinear problems governed by the $p$-Laplacian operator, cf. also \cite{PS1} for ground state solutions.  Furthermore, the case of exponential nonlinearities is treated in \cite{BCN,PS2}.

The aim of the present paper is to show that the above multiplicity pattern can still be provided for the strongly nonlinear problem \eqref{eq:main}: even more, due to the singular character of the Minkowski-curvature operator, the existence of solutions with $u(R_1) > s_0$ does not require any sub-criticality condition. Indeed, we work with a minimal set of assumptions for $f$; precisely, we just assume
\begin{itemize}
\item[$(f_\textrm{reg})$] $f \in \mathcal{C}^1([0,+\infty))$;
\item[$(f_\textrm{eq})$] $f(0) = f(s_0) = 0$, $f(s) < 0$ for $0 < s < s_0$ and $f(s) > 0$ for $s > s_0$.
\end{itemize}

In order to state our main result, we need to introduce the radial eigenvalue problem for the Laplacian with homogeneous Neumann boundary conditions in $\Omega$, that is
\begin{equation}\label{eq:eigenv_radial}
-(r^{N-1} u')'=\lambda r^{N-1} u \quad\text{in } (R_1,R_2) \qquad u'(R_1)=u'(R_2)=0.
\end{equation}
It is well known that all the eigenvalues are simple and arranged in an increasing sequence
\[
0=\lambda_1^{\mathrm{rad}}<\lambda_2^{\mathrm{rad}}<\ldots<\lambda_k^{\mathrm{rad}}<\ldots\to+\infty
\]
(for additional information concerning this problem see Appendix \ref{sec:5} ahead).
The main result of the paper is the following.

\begin{theorem}\label{thm:main}
Let $\Omega$ be either the annulus $\mathcal A(R_1,R_2)$ or the ball $\mathcal B(R_2)$ and let $f$ satisfy $(f_\textrm{reg})$, $(f_\textrm{eq})$ and
\begin{itemize}
\item[$(f_{s_0})$] for some integer $k \geq 1$ it holds $f'(s_0) > \lambda_{k+1}^{\textnormal{rad}}$.
\end{itemize}
Then there exist at least $2k$ distinct non-constant radial solutions $u_1,\ldots,u_{2k}$ to \eqref{eq:main}. 
Moreover, we have
\begin{itemize}
\item[(i)] $u_j(R_1)<s_0$ for every $j=1,\ldots,k$;
\item[(ii)] $u_j(R_1)>s_0$ for every $j=k+1,\ldots,2k$;
\item[(ii)] $u_j(r)-s_0$ and $u_{j+k}(r)-s_0$ have exactly $j$ zeros for $r\in(R_1,R_2)$, for every $j=1,\ldots,k$.
\end{itemize}
\end{theorem}

For the proof of Theorem \ref{thm:main}, similarly as in the papers \cite{ABF2,ABF,MontefuscoPucci} we use a shooting approach for the equivalent ODE problem
\begin{equation}\label{eq:main_radial}
\begin{cases}
\left(r^{N-1} \frac{u'}{\sqrt{1-(u')^2}}\right)'+r^{N-1} f(u)=0 \qquad r\in (R_1,R_2)\\
u > 0 \\
u'(R_1)=u'(R_2)=0.
\end{cases}
\end{equation}
Precisely, we first write the equation as the equivalent first order system
\begin{equation}\label{sys:intro}
u' = \frac{v}{r^{N-1}\sqrt{1+ (v/r^{N-1})^2}}, \qquad v' = - r^{N-1} f(u),
\end{equation}
and we then look for values $d \in (0,+\infty) \setminus \{s_0\}$ such that the solution
$(u_d,v_d)$ with initial condition $(u_d(R_1),v_d(R_1)) = (d,0)$ satisfies $v_d(R_2) = 0$. Due to the assumption $(f_\textrm{eq})$, solutions $(u_d,v_d)$ wind around the equilibrium point $(s_0,0)$ in the clockwise sense and, actually, the number of half-turns around such a point is nothing but the number of intersections of $u_d$ with the constant solution $u \equiv s_0$ (see Figure \ref{fig:polar_coord}).
By estimating, via assumption $(f_{s_0})$, the number of half-turns around $(s_0,0)$ when $d \to s_0$ and the number of half-turns when $d \to 0^+$ and $d \to +\infty$, the existence of solutions making a precise number of half-turns follows from a continuity argument, cf. Fig. \ref{fig:num_giri}.
\begin{figure}[h]
\includegraphics[scale=0.8]{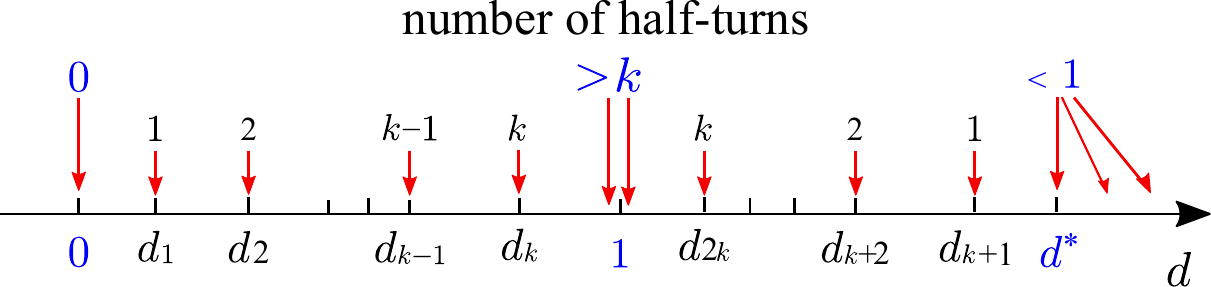}
\caption{Number of half-turns performed by the solutions of the Cauchy problem $(u(R_1),v(R_1))=(d,0)$ associated with \eqref{sys:intro} varying with the initial condition $u(R_1)=d$. The existence of a bound from above $d^*$ for the initial data $d$ in correspondence to which the solutions of the Cauchy problem perform at least one half-turn in the phase plane is a consequence of the fact that $|u'|\le 1$, see \eqref{sys:intro}.}\label{fig:num_giri}
\end{figure}
It is interesting to observe that the singular character of the Minkowski operator reflects into the fact that the right-hand side in the equation for $u'$ (see \eqref{sys:intro}) is globally bounded, this being ultimately the reason why no-subcriticality conditions are needed in this setting.
On the contrary, as already mentioned, when the equation is governed by the Laplacian (cf. \cite{bonheure2016multiple}) or the $p$-Laplacian operator (cf. \cite{ABF2,CNproc}), the existence of solutions having $u(R_1)>s_0$ could be proved only if the nonlinearity $f$ is Sobolev-subcritical. Indeed, this additional assumption allows to prove a priori estimates on the solutions, which are crucial in the proof of a bound $d^*$ from above of initial data $d$ and eventually in the proof of the existence of solutions with $u(R_1)>s_0$. The necessity of such a growth condition is also confirmed by numerical evidence in \cite[Fig. 16]{bonheure2016multiple}. 

In the case of the prototype nonlinearity \eqref{eq:prototype}, Theorem~\ref{thm:main} reduces to the following result. 

\begin{corollary}\label{cor:prot}
Let $\Omega$ be either the annulus $\mathcal A(R_1,R_2)$ or the ball $\mathcal B(R_2)$ and let $f(s)=s^{q-1}-s^{r-1}$ with $r\geq2$ and $q-r>\lambda_{k+1}^{\mathrm{rad}}$ for some integer $k\ge 1$.
Then there exist at least $2k$ distinct non-constant radial solutions $u_1,\ldots,u_{2k}$ to \eqref{eq:main}. 
Moreover, we have
\begin{itemize}
\item[(i)] $u_j(R_1)<1$ for every $j=1,\ldots,k$;
\item[(ii)] $u_j(R_1)>1$ for every $j=k+1,\ldots,2k$;
\item[(ii)] $u_j(r)-1$ and $u_{j+k}(r)-1$ have exactly $j$ zeros for $r\in(R_1,R_2)$, for every $j=1,\ldots,k$.
\end{itemize}
\end{corollary}

To illustrate better our results stated in Theorem \ref{thm:main} and Corollary \ref{cor:prot}, we present below some numerical simulations performed with the software AUTO-07P \cite{AUTO}. In the simulations we consider the following special case of \eqref{eq:main_radial}:
\begin{equation}\label{eq:rad_prot}
\begin{cases}
\left(r^{N-1} \frac{u'}{\sqrt{1-(u')^2}}\right)'+r^{N-1} (-u^2+u^{q-1})=0 \qquad r\in (0,1)\\
u > 0 \\
u'(0)=u'(1)=0,
\end{cases}
\end{equation}
with $q>3$ and $N=1,2$. In Figures \ref{fig:dim1} (a) and \ref{fig:dim2} (a) below, we represent $u(0)$ as a function of $q$. The black horizontal line corresponds to the constant solution $u\equiv s_0=1$. We can see that, as soon as the exponent $q$ overcomes a value of the form $3+\lambda_{k+1}^{\mathrm{rad}}$ ($k\ge1$), a new branch of solutions appears and each of these branches bifurcates from the constant solution. Furthermore, in Figures \ref{fig:dim1} (b) and \ref{fig:dim2} (b) we plot a selection of nonconstant solutions, in order to show their oscillatory behavior. It can be seen that solutions belonging to different branches present different oscillatory behaviors, so that the appearance of a new branch of solutions indicates the existence of a new solution with a different oscillatory behavior. This is coherent with the multiplicity result stated in Corollary \ref{cor:prot}.
We can finally observe from the simulations (a), that while in dimension $N=1$ the values $3+\lambda_{k+1}^{\mathrm{rad}}$ seem to be optimal in $q$ for the existence of a new solution having a different oscillatory behavior, in dimension $N=2$ the bifurcation points are of different type (called {\it transcritical}) and so, for instance, a nonconstant (decreasing) solution can be expected to exist also in correspondence to values of $q$ that are very close to, but {\it smaller} than $3+\lambda_2^{\mathrm{rad}}$. An analogous situation has been detected and analytically proved via bifurcation theory, for the semilinear problem in \cite{bonheure2016multiple}. It would be interesting to perform a bifurcation analysis also in the quasilinear case and to prove analytically the numerical evidence that the branches of nonconstant solutions actually bifurcate from the branch of constant solutions $u\equiv s_0$.
\begin{figure}[h!]
\includegraphics[width=\textwidth]{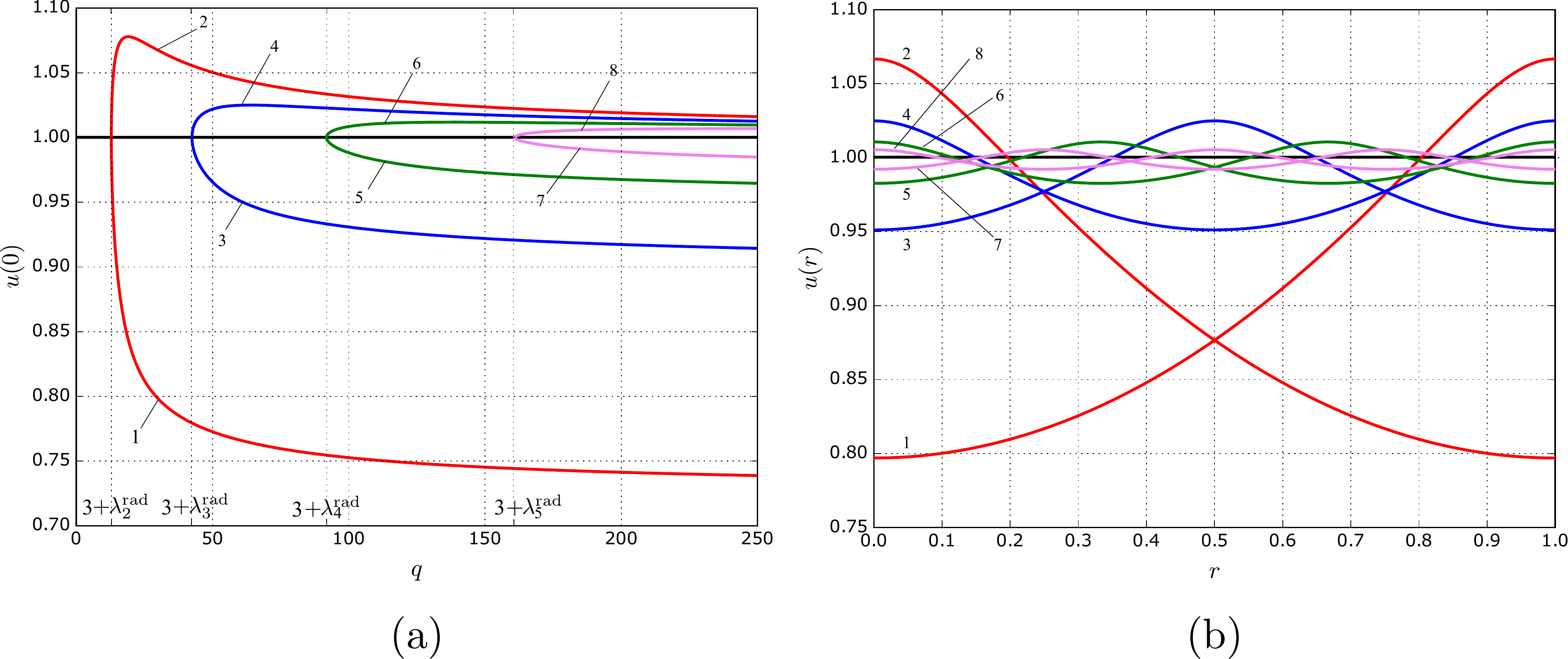}
\caption{(a) Partial bifurcation diagram in dimension $N=1$, with $R_1=0$, $R_2=1$, and $s_0=1$. (b) Graphs of eight solutions belonging to the four branches represented in (a). The colour of each solution is the same as the branch it belongs to. For each branch we have selected two solutions, one with $u(0)>1$ and the other with $u(0)<1$. The solutions displayed correspond to different values of $q$.}\label{fig:dim1}
\end{figure}
\begin{figure}[h!]
\includegraphics[width=\textwidth]{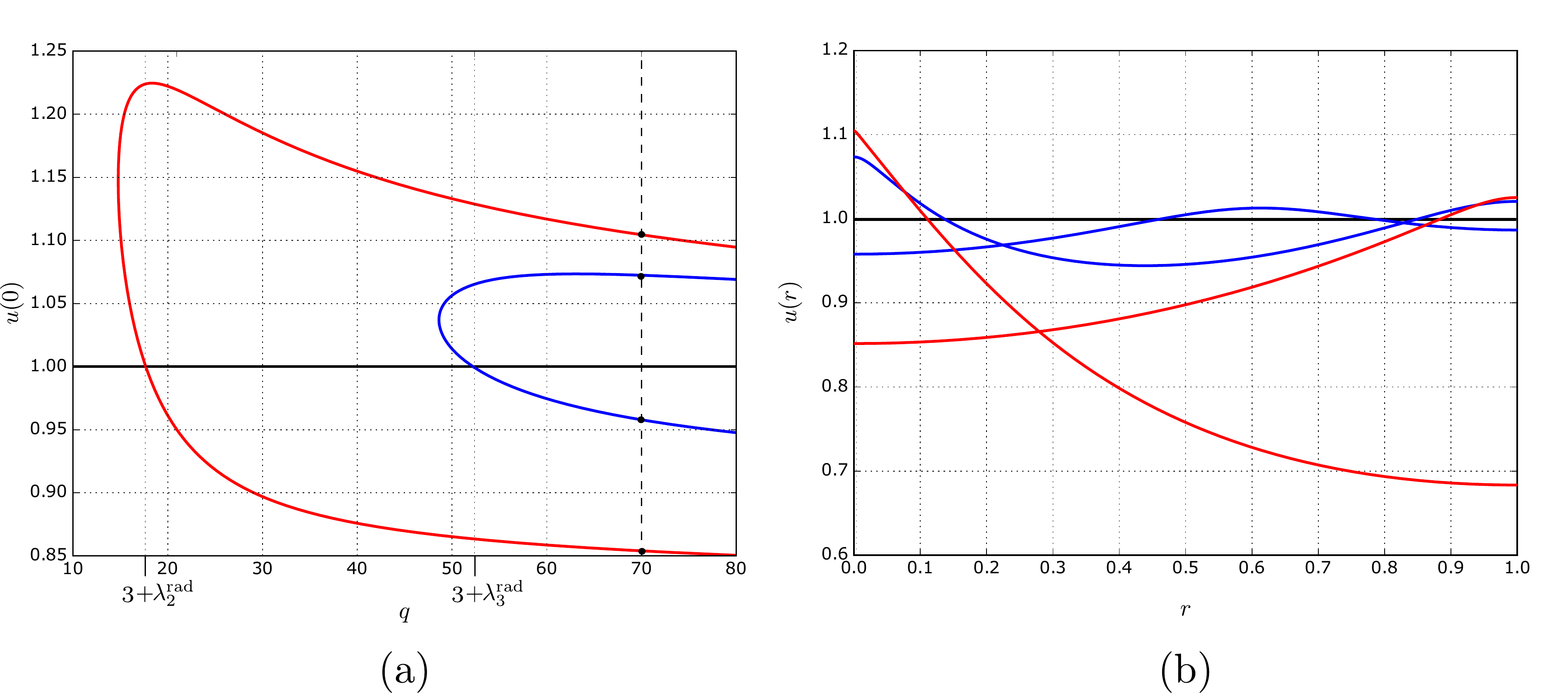}
\caption{(a) Partial bifurcation diagram in a unit disk (i.e., $R_1=0$, $R_2=1$, and $N=2$), with $s_0=1$. (b) Solutions corresponding to $q=70$.}\label{fig:dim2}
\end{figure}


The plan of the paper is the following. In Section \ref{sec:2} we prove some preliminary results for the (possibly singular) Cauchy problem $(u(R_1),v(R_1)) = (d,0)$ associated with \eqref{sys:intro}. In Section \ref{sec:3} we estimate, using a trick based on a system of scaled polar coordinates, the number of half-turns around the equilibrium point $(s_0,0)$ of solutions $(u_d,v_d)$ with $d \to s_0$. Finally, in Section \ref{sec:3} we give the proof of Theorem~\ref{thm:main}. 
The paper ends with Appendix \ref{sec:5}, where we recall some known results for the radial Neumann eigenvalue problem for $-\Delta u = \lambda u$.

\section{The Cauchy problem}\label{sec:2}
In all the paper we will assume that $\Omega$ is either the annulus $\mathcal A(R_1,R_2)$ or the ball $\mathcal B(R_2)$ and that $f$ satisfies $(f_\textrm{reg})$ and $(f_\textrm{eq})$.

Let us introduce the trivial extension of $f$
\begin{equation}\label{eq:f_hat}
\hat f(s):=\begin{cases}
	f(s) \quad&\text{if } s \geq 0\\
	0 &\text{if } s < 0,
\end{cases}
\end{equation}
which is continuous thanks to $(f_\textrm{eq})$.
We consider the radial problem
\begin{equation}\label{eq:radial}
\begin{cases}
(r^{N-1} \varphi(u'))'+r^{N-1}\hat f(u)=0 \qquad r\in (R_1,R_2)\\
u'(R_1)=u'(R_2)=0.
\end{cases}
\end{equation}
Here, with the usual abuse of notation, $u(r)=u(|x|)$, the prime symbol $'$ denotes the derivative with respect to $r$, and
\begin{equation}\label{eq:phi_def}
\varphi(s):=\frac{s}{\sqrt{1-s^2}}.
\end{equation}
For future use, note that $\varphi$ is invertible with inverse
\[
\varphi^{-1}(t)=\frac{t}{\sqrt{1+t^2}}
\]
and that
\begin{equation}\label{eq:inverse-phi}
|\varphi^{-1}(t)|<1 \qquad\text{for all } t\in\R.
\end{equation}
We shall see (cf. Lemma \ref{lem:hat_f} at the end of this section) that $u$ is a nonconstant radial solution of \eqref{eq:main} if and only if $u$ is a nonconstant solution of \eqref{eq:radial}. 

Before doing that, we prove uniqueness, global continuability and regularity for the associated Cauchy problem; hence we consider, for every $d\geq 0$,
\begin{equation}\label{eq:shooting}
\begin{cases}
u'=\varphi^{-1}\left(\frac{v}{r^{N-1}}\right) \qquad & r\in (R_1,R_2)\\
v'=-r^{N-1} \hat f(u) \qquad & r\in (R_1,R_2) \\
u(R_1)=d \\
v(R_1)=0.
\end{cases}
\end{equation}
We note that, despite the presence of the term $v/r^{N-1}$ (which is singular if $N \geq 2$ and $R_1 = 0$), using \eqref{eq:inverse-phi} it follows that the right hand side of the above system, that is $F:(R_1,R_2)\times\R^2\to \R^2$ defined as
\[
F(r,(u,v)) :=
\left(\varphi^{-1}\left(\frac{v}{r^{N-1}}\right), -r^{N-1} \hat f(u)\right),
\]
is an $L^\infty$-Carath\'eodory function, namely
\begin{itemize}
\item[(i)] $F(\cdot,(u,v))$ is measurable for every $(u,v)\in \R^2$;
\item[(ii)] $F(r,\cdot)$ is continuous for almost every $r\in (R_1,R_2)$;
\item[(iii)] for every $K>0$ there exists $C_K>0$ such that $|F(r,(u,v))|\leq C_K$ for almost every $r \in (R_1,R_2)$ and for every $(u,v)\in\R^2$ with $|(u,v)|\leq K$.
\end{itemize}
Hence, by the Peano existence theorem for ODEs with $L^\infty$-Carath\'eodory right hand side (see for example \cite[Section 1.5]{Hale}), the existence of a local $W^{1,\infty}$ solution of \eqref{eq:shooting} is guaranteed. Below, we prove uniqueness, higher regularity and global continuability for such a solution, henceforth denoted by $(u_d,v_d)$, as well as a continuous dependence result.

\begin{lemma}\label{lem:uniqueness_Cauchy}
For every $d\geq 0$, the local $W^{1,\infty}$ solution $(u_d,v_d)$ of \eqref{eq:shooting} is unique and can be defined on the whole $[R_1,R_2]$; moreover, $u_d$ is of class $C^2([R_1,R_2])$, with $u_d'(R_1) = 0$.

In addition, if $(d_n)\subset [0,+\infty)$ is such that $d_n\to d\in [0,+\infty)$ as $n\to+\infty$, then 
\begin{equation}\label{eq:dip_cont}
(u_{d_n}(r),v_{d_n}(r))\to (u_d(r),v_d(r))\quad \mbox{uniformly for }r\in [R_1,R_2],
\end{equation}
\begin{equation}\label{eq:dip_cont_u'}
u'_{d_n}(r)\to u'_d(r) \quad \mbox{uniformly for }r\in [R_1,R_2].
\end{equation}
\end{lemma}
\begin{proof} We first observe that any solution $(u_d,v_d)$ can be defined in $[R_1,R_2]$. Indeed, since $|u'_d|<1$ by \eqref{eq:inverse-phi} and \eqref{eq:shooting}, we obtain 
$$
|u_d(r)|\le d+\int_{R_1}^r |u_d'(s)|ds \le d+R_2-R_1
$$
and consequently
$$
|v_d(r)|\le  \int_{R_1}^r s^{N-1}|\hat f(u_d(s))|ds \le \frac{(R_2-R_1)^N}{N} \max_{|s|\le d+R_2-R_1}|\hat f(s)|
$$
for every $\textcolor{blue}{R_1\le}r\le R_2$ in the maximal interval of definition of the solution $(u_d,v_d)$. 
Hence, $(u_d,v_d)$ cannot blow-up in a finite interval, proving that it can be globally extended in $[R_1,R_2]$.

In the rest of the proof we suppose $R_1=0$, since the result is standard in case $R_1>0$.
In order to prove uniqueness, let $(u_d,v_d)$ and $(\bar u_d,\bar v_d)$ be two distinct solutions of \eqref{eq:shooting} and define $(U,V):=(\bar u_d-u_d,\bar v_d-v_d)$. Then $(U,V)$ is a solution of the following problem 
\begin{equation}\label{pb:U'V'}
\begin{cases}
&U'=\varphi^{-1}\left(\frac{\bar v_d}{r^{N-1}}\right)-\varphi^{-1}\left(\frac{v_d}{r^{N-1}}\right)\\
&V'=-r^{N-1}(\hat f(\bar u_d)-\hat f(u_d))\\
&U(0)=V(0)=0.
\end{cases}
\end{equation}
Using the local-Lipschitz continuity of $\varphi^{-1}$ and $\hat f$ we find
\begin{equation*}
|U'(r)|\le K \frac{|V(r)|}{r^{N-1}} \quad\mbox{and}\quad |V'(r)|\le K r^{N-1}|U(r)|
\end{equation*}
for a.e. $r\in [0,R_2]$, where $K > 0$ is a suitable constant.
Then,  
$$
\begin{aligned}
|U(r)|&\le \int_{0}^r |U'(s)| ds\le K\int_0^r  \frac{|V(s)|}{s^{N-1}}ds\le
K\int_{0}^r \frac{1}{s^{N-1}}\left(\int_0^s |V'(t)|dt\right) ds \\
&\le 
K^2 \int_{0}^r \frac{1}{s^{N-1}}\left(\int_0^s t^{N-1}|U(t)|dt\right) ds
\le K^2 \int_{0}^r \left(\int_0^s|U(t)|dt\right) ds\\
& \le K^2 R_2 \int_{0}^r |U(t)|dt
\end{aligned}
$$
for every $r\in [0,R_2]$.
Hence, by Gronwall's Lemma, 
$$
U(r)=0\quad	\mbox{for every }r\in [0,R_2]
$$
and consequently, using \eqref{pb:U'V'}, also $V(r)=0$ for every $[0,R_2]$. 

We now prove that $u_d$ is of class $C^1([0,R_2])$, with $u_d'(0) = 0$. By the second equation in \eqref{eq:shooting}, we obtain
$$
\left\vert \frac{v_d(r)}{r^{N-1}}\right \vert = \left\vert \int_{0}^r \left( \frac{s}{r}\right)^{N-1} \hat f(u_d(s))\,ds\right \vert
\leq \int_{0}^r \vert \hat f(u_d(s)) \vert\,ds.
$$
Therefore, 
$$
\frac{v_d(r)}{r^{N-1}} \to 0, \quad \mbox{ as } r \to 0^+,
$$
and consequently, by the continuity of $\varphi^{-1}$, we conclude that $u_d'$ can be continuously extended up to $r = 0$ with $u_d'(0) = 0$.

In order to show that $u_d$ is of class $C^2([0,R_2])$, following \cite[Remark 3.3]{CoelhoCorsatoRivetti} let us prove that
\begin{equation}\label{eq:u_d_C2}
\lim_{r\to0^+} \frac{\varphi(u'_d(r))}{r}=-\frac{\hat f(d)}{N}.
\end{equation}
To this aim, fix $\varepsilon>0$. By the continuity of $\varphi$ and of $u_d$, there exists $\delta>0$ such that $|\hat f(d)-\hat f(u_d(s \textcolor{blue}{r}))|<\varepsilon$ for every $s\textcolor{blue}{r}\in[0,\delta)$. Taking $r\in (0,\delta)$, we have
\begin{multline*}
\left| \frac{\hat f(d)}{N} +\frac{\varphi(u_d'(r))}{r}\,\right|=
\left| \frac{\hat f(d)}{N} -\frac{1}{r^N}\int_{0}^r s^{N-1} \hat f(u_d(s))\,ds\,\right| \\
=\left| \frac{1}{r^N} \int_0^r s^{N-1} \left( \hat f(d)-\hat f(u_d(s)) \right) \,ds \,\right| 
\leq \frac{\varepsilon}{N},
\end{multline*}
so that \eqref{eq:u_d_C2} holds. As a consequence, $\varphi(u'_d)$ is of class $C^1([0,R_2])$, which implies, being $\varphi^{-1}$ regular, that $u_d$ is of class $C^2([0,R_2])$.

Concerning the continuous dependence, the proof of \eqref{eq:dip_cont} is a standard argument based on the Ascoli-Arzel\`a Theorem. 
We now prove \eqref{eq:dip_cont_u'}. Let $d_n\to d$, then, by the second equation in \eqref{eq:shooting}, we get for every $r\in [R_1,R_2]$
\begin{multline}
\left\vert \frac{v_{d_n}(r)-v_d(r)}{r^{N-1}}\right \vert = \left\vert \int_{0}^r \left( \frac{s}{r}\right)^{N-1} [\hat f(u_{d_n}(s))-\hat f(u_d(s))]\,ds\right \vert\\
\leq \int_{0}^r \vert \hat f(u_{d_n}(s))-\hat f(u_d(s)) \vert\,ds \to 0 \quad\mbox{as } n\to\infty,
\end{multline}
where in the last step we used \eqref{eq:dip_cont} and the continuity of $\hat f$.
Finally, by the first equation in \eqref{eq:shooting} and the continuity of $\varphi^{-1}$, \eqref{eq:dip_cont_u'} follows. 
\end{proof}

\begin{remark}
Let us notice that, in the case $R_1 = 0$, the above lemma implies that $v_d(r)/r^{N-1}$ can be continuously extended \textcolor{blue}{to 0} up to $r = 0$ and that
(cf. \eqref{eq:dip_cont_u'})
\begin{equation}\label{eq:dip_cont_vr}
\frac{v_{d_n}(r)}{r^{N-1}}\to \frac{v_d(r)}{r^{N-1}} \quad \mbox{uniformly for }r\in [0,R_2].
\end{equation}
\end{remark}
We conclude the section with a maximum-type principle.

\begin{lemma}\label{lem:hat_f} The function
$u$ is a radial solution of \eqref{eq:main} if and only if $u$ solves \eqref{eq:radial} and $u\not\equiv -C$ with $C\geq0$.
\end{lemma}
\begin{proof}
Proceeding exactly as in \cite[Lemma 2.1]{ABF}, one can prove that if $u$ solves \eqref{eq:radial} and $u\not\equiv -C$ with $C\geq0$ then $u\geq0$.
It only remains to show that $u(r)>0$ for $r\in [R_1,R_2]$. To this aim, suppose by contradiction that there exists $\bar r \in [R_1,R_2]$ such that $u(\bar r)=0$, then also $u'(\bar r)=0$. If $\bar r>0$, the standard Cauchy-Lipschitz theory implies that $u\equiv0$, which is a contradiction. Otherwise $\bar r=0$, and we proved in Lemma \ref{lem:uniqueness_Cauchy} that also in this case $u\equiv0$.
\end{proof}

\section{Scaled polar coordinates}\label{sec:3}
Thanks to the uniqueness result proved in Lemma \ref{lem:uniqueness_Cauchy}, we can study system \eqref{eq:shooting} by passing to scaled polar coordinates around the point $(s_0,0)$. For $\alpha>0$ and $r\in [R_1,R_2]$, let
\begin{equation}\label{eq:polar}
\begin{cases}
u(r)-s_0=\rho(r)\cos\theta(r) \\
v(r)=-\alpha \rho(r)\sin\theta(r).
\end{cases}
\end{equation}
For $d\geq 0$ and $d\neq s_0$, if $(u_d,v_d)$ satisfies \eqref{eq:shooting}, the corresponding $(\theta_d,\rho_d)$ is such that, for $r\in (R_1,R_2)$,
\begin{equation}\label{eq:theta'}
\begin{aligned}
\theta_d'&=\frac{1}{\alpha\rho_d^2} \left[\varphi^{-1}\left(\frac{v_d}{r^{N-1}}\right)v_d + r^{N-1}\hat f(u_d)(u_d-s_0) \right] \\
&=\frac{\alpha \sin^2\theta_d}{r^{N-1}[1+(v_d/r^{N-1})^2]^{1/2}}
+r^{N-1} \hat f(u_d) \frac{u_d-s_0}{\alpha\rho_d^2}
\end{aligned}
\end{equation}
with initial conditions
\[
\theta_d(R_1)=\begin{cases}
\pi \quad \text{if } 0<d<s_0\\
0 \quad \text{if } d>s_0.
\end{cases}
\quad\text{and}\quad
\rho_d(R_1)=|d-s_0|.
\]

We remark that, by \eqref{eq:theta'}, $\theta_d(r)$ is strictly increasing for every $r\in [R_1,R_2]$, so that the solution $(u_d,v_d)$ is turning clockwise around $(s_0,0)$ in the phase plane $(u,v)$.

\begin{figure}[h]
\includegraphics{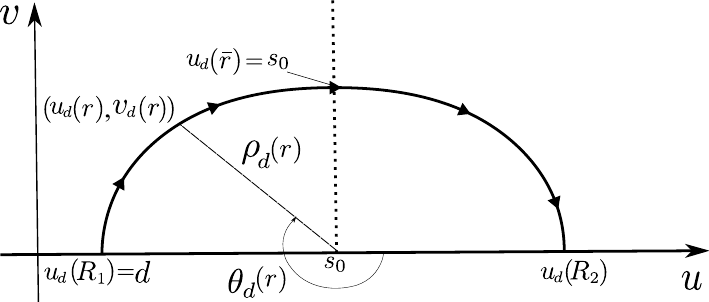}
\caption{A solution $(u_d,v_d)$, with $0<d<s_0$, in the phase plane $(u,v)$. The solution is also given in polar coordinates $(\theta_d,\rho_d)$ with $\alpha=1$. It can be noted from the picture that $u_d(\bar r)=s_0$ if and only if $\cos\theta_d(\bar r)=0$ and that $v_d(r)=0$ for some $r\in [R_1,R_2]$ if and only if $\sin\theta_d(r)=0$.}\label{fig:polar_coord}
\end{figure}

We would like to estimate the quantity $\theta_d(R_2) - \theta_d(R_1)$ for $d$ in a neighborhood of $s_0$. We remark that, despite the fact that the angle $\theta_d$ depends on the constant $\alpha$, the quantity
\[
\left\lfloor\frac{\theta_d(R_2) - \theta_d(R_1)}{\pi}\right\rfloor,
\]
that is the number of half turns of the solution around $(s_0,0)$, does not depend on $\alpha$. Here $\lfloor\cdot\rfloor$ denotes the floor function. Indeed, this quantity is the number of zeros of $v(r)$ for $r\in (R_1,R_2]$, which clearly does not depend on $\alpha$, cf. also Fig. \ref{fig:polar_coord}.

\begin{lemma}\label{lem:d_close1}
Suppose that, for some integer $k\geq1$,
\[
f'(s_0)>\lambda_{k+1}^{\textnormal{rad}}.
\]
There exists $\bar\delta>0$ such that $\theta_d(R_2) - \theta_d(R_1)>k\pi$ for $d\in [s_0-\bar\delta,s_0+\bar\delta]$ and $d \neq s_0$.
\end{lemma}
\begin{proof}
Suppose that $f'(s_0)>\lambda_{k+1}^{\mathrm{rad}}$.
Let $\bar\lambda$ be such that
\begin{equation}\label{eq:bar_lambda_def}
\lambda_{k+1}^{\mathrm{rad}}<\bar \lambda<f'(s_0),
\end{equation}
and consequently choose $\varepsilon>0$ such that
\begin{equation*}
\bar\lambda \sqrt{1+\varepsilon^2} <f'(s_0)-\varepsilon.
\end{equation*}
Then, using assumptions $(f_\textrm{reg})$ and $(f_\textrm{eq})$, there exists $\delta>0$ such that, for every $s$ satisfying $|s-s_0|\leq \delta$, it holds
\begin{equation}\label{eq:ineq1}
\hat f(s)(s-s_0)\geq (f'(s_0)-\varepsilon)(s-s_0)^2 \geq
\bar\lambda \sqrt{1+\varepsilon^2} (s-s_0)^2.
\end{equation}
Thanks to \eqref{eq:dip_cont} and \eqref{eq:dip_cont_vr}, there exists $\bar\delta>0$ such that, for every $d$ satisfying $0<|d-s_0|\leq\bar\delta$, it holds
\begin{equation}\label{eq:ineq2}
|u_d(r)-s_0|\leq \delta \qquad\text{and}\qquad 
\frac{|v_d(r)|}{r^{N-1}}\leq\varepsilon,
\end{equation}
for every $r\in [R_1,R_2]$, being $u_{s_0}\equiv s_0$ in $[R_1,R_2]$. Now we choose in \eqref{eq:polar} and consequently in \eqref{eq:theta'}
\[
\alpha=\sqrt{1+\varepsilon^2}.
\]
By replacing \eqref{eq:ineq1} and \eqref{eq:ineq2} into \eqref{eq:theta'}, and using \eqref{eq:polar}, we obtain that, for every $d$ satisfying $0<|d-s_0|\leq \bar\delta$ and $r\in [R_1,R_2]$,
\begin{equation}
\theta_d'(r)\geq \frac{\sin^2\theta_d(r)}{r^{N-1}}+\bar\lambda r^{N-1} \cos^2\theta_d(r).
\end{equation}
Recalling equation \eqref{eq:theta_mu} with $\mu=\bar\lambda$, using the Comparison Theorem for ODEs and \eqref{eq:initial_linear}, we obtain, for all $d$ satisfying $0<|d-s_0|\leq\bar\delta$,
$$
\theta_d(r) - \theta_d(R_1)\geq \vartheta_{\bar\lambda}(r)\quad\mbox{for all }r\in[R_1,R_2].
$$
In particular, by relation \eqref{eq:bar_lambda_def} and Theorems \ref{th:RW} and \ref{th:eigen}, we have
\[
\theta_d(R_2) - \theta_d(R_1)>  \vartheta_{\lambda_{k+1}^{\mathrm{rad}}}(R_2) = k\pi.  \qedhere
\]
\end{proof}

\section{Proof of the main result}\label{sec:4}

\begin{proof}[Proof of Theorem \ref{thm:main}]
Suppose first $d\in [0,s_0)$ and let $(u_d,v_d)$ be the solution of \eqref{eq:shooting}. We consider the associated angular variable $\theta_d$ given by \eqref{eq:polar} with $\alpha=1$. Notice that $\theta_d(R_1)=\pi$. By assumption $(f_{s_0})$ and Lemma \ref{lem:d_close1}, we have
\begin{equation}\label{eq:at_s_0}
\theta_d(R_2) >( k+1)\pi \quad \text{for }d\in [s_0-\bar\delta,s_0).
\end{equation}
Furthermore 
\begin{equation}\label{eq:at_0}
\theta_0(r)\equiv\pi \quad \text{for every } r\in [R_1,R_2].
\end{equation}
By \eqref{eq:dip_cont}, the map $d\mapsto \theta_d(R_2)$ is continuous on $[0,s_0)$. Thus \eqref{eq:at_s_0} and \eqref{eq:at_0} imply that, for all $j=1,\ldots,k$, there exists $d_j \in (0,s_0)$ such that 
\begin{equation}\label{eq:d_j}
\theta_{d_j}(R_2)=(j+1)\pi.
\end{equation}
By \eqref{eq:polar}, this corresponds to $u'_{d_j}(R_2)=0$, so that $u_j:=u_{d_j}$ for $j=1,\ldots,k$ are the solutions mentioned in point (i). 
In order to study the oscillatory behavior of $u_j$ for $j=1,\ldots,k$, we notice that, by \eqref{eq:theta'} and \eqref{eq:d_j}, $\theta_{d_j}(\cdot)$ is strictly increasing from $\pi$ to $(j+1)\pi$. Hence there exist exactly $j$ radii $r_1,\ldots,r_{j}\in (R_1,R_2)$ such that $\theta_{d_j}(r_1)=\frac{3}{2}\pi_p$, $\theta_{d_j}(r_2)=\frac{5}{2}\pi_p,\ldots,\theta_{d_j}(r_{j})=\left(j+\frac{1}{2}\right)\pi_p$. Again by \eqref{eq:polar}, we have $u_j(r)=s_0$ if and only if $r\in \{r_1,\ldots,r_j\}$.

If $d>s_0$, then $\theta_d(R_1)=0$. Again by assumption $(f_{s_0})$ and Lemma \ref{lem:d_close1}, we have
\[
\theta_d(R_2) >k\pi \quad \text{for }d\in (s_0,s_0+\bar\delta].
\] 
On the other hand, we claim that, letting $d^\ast :=s_0+R_2-R_1$, it holds
\[
\theta_{d^\ast}(R_2)<\pi.
\] 
Indeed by \eqref{eq:inverse-phi} we have, for every $r\in [R_1,R_2]$,
\begin{multline*}
\rho_{d^\ast}(r)\cos\theta_{d^\ast}(r)=u_{d^\ast}(r)-s_0
\geq -\int_{R_1}^r |u'_{d^\ast}(s)| \,ds +d^\ast-s_0 \\
> -(R_2-R_1)+d^\ast-s_0=0,
\end{multline*}
implying that $\theta_{d^\ast}(r)<\pi/2$ for every $r\in [R_1,R_2]$.

As in the previous case, we infer the existence of $k$ values $d_{j+k}$ for $j=1,\ldots,k$ such that $\theta_{d_{j+k}}(R_2)=j\pi$. We thus obtain the solutions $u_{j+k}:=u_{d_{j+k}}$ for every $j=1,\ldots,k$ whose existence is stated in (ii). Similarly as before, there exist $r_1,\ldots,r_{j}\in (R_1,R_2)$ such that $\theta_{d_{j+k}}(r_1)=\frac{1}{2}\pi$, $\theta_{d_{j+k}}(r_2)=\frac{3}{2}\pi,\ldots,\theta_{d_{j+k}}(r_{j})=\left(j-\frac{1}{2}\right)\pi$, proving the oscillatory behavior stated at (iii).
\end{proof}

\appendix

\section{Radial eigenvalue problem for the Neumann Laplacian}\label{sec:5}

In this appendix we recall some known results concerning the radial eigenvalue problem for the Laplacian with homogeneous Neumann boundary conditions \eqref{eq:eigenv_radial}.
Even though this equation contains the possibly singular weight $r^{1-N}$, it is well known that the eigenfunctions satisfy the classical Sturm theory. We refer for example to \cite{RW99}, where indeed a more general problem is treated.

\begin{theorem}{\cite[Theorem 1]{RW99}}\label{th:eigen}
The eigenvalue problem \eqref{eq:eigenv_radial} has a countable number of simple eigenvalues $0=\lambda_1^{\mathrm{rad}}<\lambda_2^{\mathrm{rad}}<\lambda_3^{\mathrm{rad}}<\dots$, $\lim_{k\to+\infty}\lambda_k^{\mathrm{rad}}=+\infty$, and no other eigenvalues. The eigenfunction that corresponds to the $k$-th eigenvalue $\lambda_k^{\mathrm{rad}}$ has $k-1$ simple zeros in $(R_1,R_2)$.  
\end{theorem}

Via the change of variables 
\[
\begin{cases}
u(r)=\varrho_\lambda(r)\cos\vartheta_\lambda(r) \\
r^{N-1}u'(r)=-\varrho_\lambda(r)\sin\vartheta_\lambda(r),
\end{cases}
\]
if $u$ is an eigenfunction associated to $\lambda$ as in \eqref{eq:eigenv_radial}, the corresponding $(\vartheta_\lambda,\varrho_\lambda)$ is such that
\begin{equation}\label{eq:theta'-omogenea-associata}
\vartheta_\lambda'=\frac{\sin^2\vartheta_\lambda}{r^{N-1}}+\lambda r^{N-1}\cos^2\vartheta_\lambda, \qquad r\in [R_1,R_2].
\end{equation}
By convention, we choose an eigenfunction $u$ satisfying $u(R_1)>0$, so that we can assume
\begin{equation}\label{eq:initial_linear}
\vartheta_\lambda(R_1)=0.
\end{equation}

Notice that, by \eqref{eq:theta'-omogenea-associata}, the function $r \mapsto \vartheta_\lambda(r)$ is strictly increasing. As a consequence, if $\lambda=\lambda^{\mathrm{rad}}_{k+1}$ for $k\ge0$, the fact that the eigenfunction which corresponds to the $(k+1)$-th eigenvalue has $k$ simple zeros in $(R_1,R_2)$ reads as 
\begin{equation}\label{initialcondition-ep}
\vartheta_{\lambda^{\mathrm{rad}}_{k+1}}(R_2) =k\pi.
\end{equation}

We will also need to consider the Cauchy problem associated to the equation \eqref{eq:theta'-omogenea-associata} when $\lambda$ is not an eigenvalue. Also in this case it is known that, despite the fact that the equation may be singular, there exists a unique solution satisfying the initial condition $\vartheta_\mu(R_1)=0$.
\begin{theorem}{\cite[Theorem 4]{RW99}}\label{th:RW}
For every $\mu>0$, let $\vartheta_\mu$ solve
\begin{equation}\label{eq:theta_mu}
\vartheta_\mu'=\frac{\sin^2\vartheta_\mu}{r^{N-1}}+\mu r^{N-1}\cos^2\vartheta_\mu, \qquad r\in [R_1,R_2],
\end{equation}
with initial condition $\vartheta_\mu(R_1)=0$.
The function $\vartheta_\mu(R_2)$ is strictly increasing in $\mu$.
\end{theorem}

\section*{Acknowledgments}
\noindent A. Boscaggin and B. Noris acknowledge the support of the project ERC Advanced Grant 2013 n. 339958: ``Complex Patterns for Strongly Interacting Dynamical Systems --COMPAT''.


\end{document}